\documentclass[a4paper,draft,10pt]{amsproc}
\usepackage{amssymb}
\usepackage{amscd} 
\usepackage{cmdtrack} 
\theoremstyle{plain}
 \newtheorem{thm}{Theorem}[section]
 \newtheorem{prop}{Proposition}[section]
 \newtheorem{lem}{Lemma}[section]
 \newtheorem{cor}{Corollary}[section]
\theoremstyle{definition}
 \newtheorem{exm}{Example}[section]
 \newtheorem{dfn}{Definition}[section]
\theoremstyle{remark}
 
 \numberwithin{equation}{section}

\renewcommand{\leq}{\leqslant}
\renewcommand{\geq}{\geqslant}

\title[Composition operators]
{Composition operators acting on weighted Hilbert spaces of analytic functions}
\subjclass[2010]{Primary 47B33; Secondary 30H30, 46E40.}

\keywords{Composition operators, Weighted analytic space, Hilbert-Schmidt, Schatten-class, Fredholm.}

\author[Hassanlou
]{\textsc{Mostafa Hassanlou$^*$
 } \\
\\
\textit{\footnotesize Depatment of Pure Mathematics, Faculty of Mathematical Sciences \\
University of Tabriz, Tabriz, Iran.\\
m$_{-}$hasanloo@tabrizu.ac.ir 
}}
\thanks{$^*$Corresponding author}

\begin{document}

\vspace{18mm}
\setcounter{page}{1}
\thispagestyle{empty}

\begin{abstract}
In this paper, we consider composition operators on weighted Hilbert spaces of analytic functions and  observe that a formula for the  essential
norm, give a Hilbert-Schmidt characterization and characterize the membership in Schatten-class for these operators. Also, 
closed range composition operators  are investigated.
\end{abstract}

\maketitle

\section{\bf Introduction}
\indent
Let $\mathbb{D}$ denotes the open unit disk $\{ z \in \mathbb{C}: |z|< 1\}$ and $\varphi$ be an analytic self map of $\mathbb{D}$. The
composition operator $C_{\varphi}$ induced by $\varphi$ is defined  $C_{\varphi} f = f \circ \varphi$, for any $f \in H(\mathbb{D})$,
the space of all analytic functions on $\mathbb{D}$. This operator can be generalized to the weighted composition operator $u C_{\varphi}$,
$u C_{\varphi} f(z) = u(z) f(\varphi(z))$, $u \in H(\mathbb{D})$.
We consider a \textit{weight} as a  positive integrable function
$\omega \in C^2 [0,1)$ which is radial, $\omega(z) = \omega(|z|)$. The weighted Hilbert space of analytic functions $\mathcal{H}_{\omega}$ consists of
all analytic functions on $\mathbb{D}$ such that
$$ ||f'||_{\omega}^2 = \int_{\mathbb{D}} |f'(z)|^2 \omega(z) \ dA(z) < \infty, $$
equipped with the norm $||f||_{\mathcal{H}_{\omega}}^2 = |f(0)|^2 + ||f'||_{\omega}^2 $.
Here $dA$ is the normalized area measure on $\mathbb{D}$.
Also the weighted Bergman spaces defined by
$$ \mathcal{A}_{\omega}^2 = \left \{ f \in H(\mathbb{D}): ||f||_{{\omega}}^2 = \int_{\mathbb{D}} |f(z)|^2 \omega(z) \ dA(z) < \infty \right \}.  $$
If $f(z) = \sum_{n=0}^{\infty} a_n z^n$, then $f \in \mathcal{H}_{\omega}$ if and only if
$$ ||f||_{\mathcal{H}_{\omega}}^2 = \sum_{n=0}^{\infty} |a_n|^2 \omega_n < \infty, $$
where $\omega_0 =1$ and for $n \geq 1$
$$ \omega_n = 2 n^2 \int_{0}^{1} r^{2n-1} \omega(r) dr, $$
and $f \in \mathcal{A}_{\omega}$ if and only if
$$ ||f||_{\mathcal{A}_{\omega}}^2 = \sum_{n=0}^{\infty} |a_n|^2 p_n < \infty, $$
where
$$ p_n = 2\int_{0}^{1} r^{2n+1} \omega(r) dr, \ \ \ n\geq 0. $$
By letting $\omega_{\alpha} (r) = (1-r^2)^{\alpha}$ (standard weight), $\alpha > -1$, $\mathcal{H}_{\omega_{\alpha}} = \mathcal{H}_{\alpha}$. If
$0 \leq \alpha <1$, then $\mathcal{H}_{\alpha} = \mathcal{D}_{\alpha}$, the weighted Dirichlet space, and  $\mathcal{H}_{1} = H^2$, the Hardy space. \\
\indent
There are several papers that studied composition operators on various spaces of analytic functions. The best monographs for these operators are
\cite{cown, shap}. In \cite{kellay1}, Kellay and Lef\`{e}vre studied composition operators on weighted Hilbert space of analytic functions by using
generalized Nevanlinna counting function. They characterized boundedness and compactness of these operators. Pau and P\'{e}rez \cite{pau} studied
boundedness, essential norm,  Schatten-class and closed range properties of these operators acting on  weighted Dirichlet spaces. \\
\indent
Our aim in this paper is to generalize the results of \cite{pau} to a large class of spaces.
Throughout the remainder of this paper, $c$ will denote a positive constant, the exact
value of which will vary from one appearance to the next.
\section{\bf Preliminaries}
In this section we give some notations and lemmas will be used in our work.
\begin{dfn} \cite{kellay1}
We assume that $\omega$ is a weight function, with the following properties \\
\indent ($W_1$): $\omega$ is non-increasing, \\
\indent ($W_2$): $\omega(r)(1-r)^{-(1+\delta)}$ is non-decreasing for some $\delta >0$, \\
\indent ($W_3$): $\lim_{r \rightarrow 1^{-}} \omega (r) =0$, \\
\indent ($W_4$): One of the two properties of convexity is fulfilled
$$\left \{
  \begin{array}{ll}
    (W_4^{(I)}):  & \omega  \ is \ convex \  and \  \lim_{r \rightarrow 1} \omega'(r) =0,  \\
  or &   \\
  (W_4^{(II)}): & \omega \ is \  concave.
  \end{array}
\right.$$
Such a weight $\omega$ is called admissible.
\end{dfn}
If $\omega$ satisfies conditions ($W_1$)-($W_3$) and ($W_4^{(I)}$) (resp. ($W_4^{(II)}$)), we shall say that $\omega$ is
(I)-admissible (resp. (II)-admissible). Also we use weights satisfy (L1) condition (due to Lusky \cite{lus}):
 $$(L1) \ \  \inf_{k} \frac{\omega(1-2^{-k-1})}{\omega(1-2^{-k})} >0. $$
This is equivalent to  this  condition (see\cite{lin}): \\
\indent There are $0 < r < 1$ and $0 < c < \infty$ with $\frac{\omega(z)}{\omega(a)} \leq c$ for every $a, z \in \Delta (a,r)$,
where $\Delta(a,r)= \{ z \in \mathbb{D}: |\sigma_a (z)| < r \}$ and
$\sigma_a (z) = \frac{a-z}{1- \overline{a}z}$ is the Mobius  transformation on $\mathbb{D}$. \\
All characterizations in this paper are needed to the generalized counting Nevanlinna function. Let
$\varphi$ be an analytic self map of $\mathbb{D}$ ($\varphi(\mathbb{D}) \subset \mathbb{D}$). The generalized counting Nevanlinna function
associated to a weight $\omega$ defined as follows
$$ N_{\varphi, \omega} (z) = \sum_{a: \varphi(a)=z} \omega (a), \ \ \  z \in \mathbb{D} \backslash \{\varphi(0)\}. $$
By using  the change of variables formula we have: If $f$ be a non-negative function on $\mathbb{D}$, then
\begin{equation} \label{r1}
\int_{\mathbb{D}} f(\varphi(z)) |\varphi'(z)|^2 \omega (z) dA(z) = \int_{\mathbb{D}} f(z) N_{\varphi, \omega} (z) dA(z).
\end{equation}
Also the generalized counting Nevanlinna function has the sub-mean value property (Lemmas 2.2 and 2.3 \cite{kellay1}). Let $\omega$ be an
admissible weight. Then for every $r>0$ and $z \in \mathbb{D}$ such that $D(z,r) \subset \mathbb{D} \backslash D(0, 1/2)$
\begin{equation} \label{r2}
 N_{\varphi, \omega} (z) \leq \frac{2}{r^2} \int_{|\zeta -z|<r} N_{\varphi, \omega} (\zeta) dA(\zeta).
\end{equation}
\begin{lem} \cite{kellay1}  \label{l2}
If  $\omega$ is a weight satisfying ($W_1$) and ($W_2$), then there exists $c>0$ such that
$$ \frac{1}{c} \omega (z) \leq \omega (\sigma_{\varphi(0)} (z)) \leq c \omega(z), \ \ \ z \in \mathbb{D}. $$
\end{lem}
\begin{lem} \cite{kellay1} \label{l1}
Let $\omega$ be a weight satisfying ($W_1$) and ($W_2$). Let $a \in \mathbb{D}$ and
$$ f_{a} (z) = \frac{1}{\sqrt{\omega(a)}} \frac{(1-|a|^2)^{1+\delta}}{(1-\overline{a} z)^{1+\delta}}. $$
Then $||f_{a}||_{\mathcal{H}_{\omega}} \asymp 1$.
\end{lem}
\section{\bf Essential Norm}
Recall that the essential norm $||T||_e$ of a bounded operator $T$ between Banach spaces $X$ and $Y$ is defined as
the distance from $T$ to $K(X,Y)$, the space of all compact operators between $X$ and $Y$.
\begin{thm} \label{g1}
Let $\omega$ be an admissible weight and $C_{\varphi}$ be a bounded operator on $\mathcal{H}_{\omega}$. Then
$$ || C_{\varphi} ||_e \asymp \limsup_{|z| \rightarrow 1^{-}} \frac{N_{\varphi, \omega} (z)}{\omega (z)}. $$
\end{thm}
\begin{proof}
Consider the test function defined in Lemma \ref{l1}. Then $\{ f_{a} \}_{a \in \mathbb{D}}$ is bounded in $\mathcal{H}_{\omega}$ and
converges uniformly on compact subsets of $\mathbb{D}$ to 0 as $|a| \rightarrow 1^-$. Then for every compact operator $K$ on $\mathcal{H}_{\omega}$,
$\lim_{|a| \rightarrow 1^-} ||K f_{a}||_{\mathcal{H}_{\omega}} =0$. There exists a constant $c >0$ such that
\begin{align*}
||C_{\varphi} - K|| \geq & c  \limsup_{|a| \rightarrow 1^{-}} ||C_{\varphi} f_{a} - K f_{a}||_{\mathcal{H}_{\omega}} \\
\geq &  c \limsup_{|a| \rightarrow 1^{-}} ||C_{\varphi} f_{a}||_{\mathcal{H}_{\omega}}.
\end{align*}
Therefore
$$ ||C_{\varphi}||_e \geq c \limsup_{|a| \rightarrow 1^{-}} ||C_{\varphi} f_{a}||_{\mathcal{H}_{\omega}}.$$
On the other hand
\begin{align*}
||C_{\varphi} f_{a}||_{\mathcal{H}_{\omega}}^2 = & |f_{a} (\varphi(0))|^2 + \int_{\mathbb{D}} |f_a' (\varphi(z))|^2 |\varphi'(z)|^2 \omega(z) dA(z) \\
=& |f_{a} (\varphi(0))|^2 + \int_{\mathbb{D}}  |f_a'(z)|^2 N_{\varphi, \omega} (z) dA(z) \\
\geq & c \frac{(1-|a|^2)^{2+2\delta} |a|^2}{\omega (a)} \int_{D(a, \frac{1-|a|}{2})} \frac{N_{\varphi, \omega} (z)}{|1-\overline{a}z|^{4+2\delta}} dA(z).
\end{align*}
If $|a|$ is close enough to 1, then $\varphi(0) \not \in D(a, \frac{1-|a|}{2})$. So $|1-\overline{a}z| \asymp (1-|a|)$ for $z \in D(a, \frac{1-|a|}{2})$.
We have
$$ \limsup_{|a| \rightarrow 1^{-}} ||C_{\varphi} f_{a}||_{\mathcal{H}_{\omega}}^2 \geq c \limsup_{|a| \rightarrow 1^{-}}
\frac{|a|^2}{\omega (a) (1-|a|)^2}  \int_{D(a, \frac{1-|a|}{2})} N_{\varphi, \omega} (z) dA(z).$$
By the sub-mean value property of $N_{\varphi, \omega}$, we get
$$\limsup_{|a| \rightarrow 1^{-}} ||C_{\varphi} f_{a}||_{\mathcal{H}_{\omega}}^2 \geq c \limsup_{|a| \rightarrow 1^{-}}
\frac{N_{\varphi, \omega} (a)}{\omega (a)}. $$
Now
$$ ||C_{\varphi}||_e^2  \geq c \limsup_{|a| \rightarrow 1^{-}} \frac{N_{\varphi, \omega} (a)}{\omega (a)},$$
and the lower estimate is obtained. The upper estimate comes from p. 136 \cite{cown}.
\end{proof}
Since the space of compact operators is a closed subspace of space of bounded operators, then a bounded operator $T$ is compact
if and only if $||T||_e =0$. According to this fact we have the following corollary which is  Theorem 1.4 \cite{kellay1}.
\begin{cor} \label{c1}
Let $\omega$ be an admissible weight. Then $C_{\varphi}$ is compact on $\mathcal{H}_{\omega}$ if and only if
$$ \lim_{|z| \rightarrow 1^{-}} \frac{N_{\varphi, \omega} (z)}{\omega (z)} =0.$$
\end{cor}
\section{\bf Hilbert-Schmidt and Schatten-class}
In this section we try to get a characterization of Hilbert-Schmidt composition operators. Another characterization
we will have as a result of Schatten-class in the case $p=2$.
\begin{thm} \label{g2}
Let $\omega$ be a weight. Then $C_{\varphi}: \mathcal{H}_{\omega} \rightarrow \mathcal{H}_{\omega}$ is Hilbert-Schmidt if and only if
$$ \int_{\mathbb{D}} ||R_z||^2 N_{\varphi, \omega} (z) dA(z) < \infty, $$
where $R_z$ is the reproducing kernel of the weighted Bergman space $\mathcal{A}_{\omega}^2$.
\end{thm}
\begin{proof}
Note that $\{ z^n / ||z^n||_{\mathcal{H}_{\omega}}\}$ is an orthonormal basis for $\mathcal{H}_{\omega}$. This implies that
\begin{align*}
\sum_{n=1}^{\infty} ||C_{\varphi} (\frac{z^n}{||z^n||_{\mathcal{H}_{\omega}}})||_{\mathcal{H}_{\omega}}^2 = &
\sum_{n=1}^{\infty} \int_{\mathbb{D}} \frac{n^2 |\varphi(z)|^{2(n-1)} }{||z^n||_{\mathcal{H}_{\omega}}^2} |\varphi'(z)|^2 \omega(z) dA(z) \\
=&  \sum_{n=1}^{\infty} \int_{\mathbb{D}} \frac{n^2 |z|^{2(n-1)}}{||z^n||_{\mathcal{H}_{\omega}}^2} N_{\varphi, \omega} (z) dA(z) \\
= & \int_{\mathbb{D}} \sum_{n=1}^{\infty}  \frac{n^2 |z|^{2(n-1)}}{\omega_n} N_{\varphi, \omega} (z) dA(z) \\
= &  \int_{\mathbb{D}} \sum_{n=1}^{\infty}  \frac{|z|^{2n}}{p_n} N_{\varphi, \omega} (z) dA(z) \\
=&  \int_{\mathbb{D}} ||R_z||^2 N_{\varphi, \omega} (z) dA(z).
\end{align*}
This completes the proof.
\end{proof}
For studying Schatten-class we need the Toeplitz operator. For more information about relation between Toeplitz operator and
Schatten-class see \cite{zhu1}. Let $\psi$ be positive function in $L^1 (\mathbb{D}, dA)$ and $\omega$ be a weight. The Toeplitz
operator associated to $\psi$ defined by
$$ T_{\psi} f(z) = \frac{1}{\omega(z)} \int_{\mathbb{D}} \frac{f(t) \psi(t) \omega(t)}{(1-\overline{z}t)^2} dA(t). $$
$T_{\psi} \in S_p (\mathcal{A}_{\omega}^2)$ if and only if the function
$$ \widehat{\psi}_r (z) = \frac{1}{(1-|z|^2)^2 \omega(z)} \int_{\Delta(z,r)} \psi(t) \omega(t) dA(t) $$
is in $L^p(\mathbb{D}, d \lambda)$, \cite{lues}, where $d \lambda = (1-|z|^2)^{-2} dA(z)$ is the hyperbolic measure on $\mathbb{D}$.
 According to the description of \cite{pau} pages 8 and 9,
$C_{\varphi} \in S_p (\mathcal{H}_{\omega})$ if and only if $\varphi'C_{\varphi} \in S_p (\mathcal{A}_{\omega}^2)$.
\begin{thm} \label{g3}
Let $\omega$ be an admissible  weight satisfy (L1) condition. Then  $C_{\varphi} \in S_p (\mathcal{H}_{\omega})$ if and only if
$$\psi (z) = \frac{N_{\varphi, \omega} (z)}{\omega(z)} \in L^{p/2} (\mathbb{D}, d \lambda).$$
\end{thm}
\begin{proof}
For any $f,g \in \mathcal{H}_{\omega} $ we have
\begin{align*}
\langle (\varphi' C_{\varphi})^*(\varphi' C_{\varphi}) f, g \rangle = &
 \langle (\varphi' C_{\varphi}) f, (\varphi' C_{\varphi})g \rangle \\
 =& \int_{\mathbb{D}} f(\varphi(z)) \overline{g(\varphi(z))} |\varphi'(z)|^2 \omega(z) dA(z)\\
 =&  \int_{\mathbb{D}} f(z) \overline{g(z)} N_{\varphi, \omega} (z) dA(z).
\end{align*}
On the other hand, since $\frac{1}{(1-\overline{z}t)^2}$ is the reproducing kernel in $\mathcal{A}^2$,
$$ T_{\psi} f(z) = \frac{1}{\omega(z)} \int_{\mathbb{D}} \frac{f(t) N_{\varphi, \omega} (t)}{(1-\overline{z}t)^2} dA(t)
= \frac{N_{\varphi, \omega} (z) f(z)}{\omega(z)}. $$
Therefore
$$ \langle T_{\psi} f, g \rangle = \int_{\mathbb{D}} f(z) \overline{g(z)} N_{\varphi, \omega} (z) dA(z). $$
Thus $T_{\psi} = (\varphi' C_{\varphi})^*(\varphi' C_{\varphi})$. Theorem 1.4.6 \cite{zhu3} implies that
$\varphi'C_{\varphi} \in S_p (\mathcal{A}_{\omega}^2)$ if and only if $(\varphi' C_{\varphi})^*(\varphi' C_{\varphi}) \in  S_{p/2} (\mathcal{A}_{\omega}^2)$.
We get $\varphi'C_{\varphi} \in S_p (\mathcal{A}_{\omega}^2)$ if and only if $T_{\psi} \in S_{p/2} (\mathcal{A}_{\omega}^2)$ if and only if
$ \widehat{\psi}_r (z) \in L^{p/2}(\mathbb{D}, d \lambda)$. \\
It is clear that $\Delta(z,r)$ contains an Euclidian disk centered at $z$ of radius $\eta (1-|z|)$ with $\eta$ depending only on $r$. By the sub-mean
value property of $N_{\varphi, \omega}$ we have
\begin{align*}
\psi (z) = \frac{N_{\varphi, \omega} (z)}{\omega(z)} \leq & \frac{2}{r^2 \omega(z)} \int_{\Delta(z,r)} N_{\varphi, \omega} (t) dA(t) \\
\leq &  \frac{2}{r^2 (1-|z|^2)^2 \omega(z)} \int_{\Delta(z,r)} \psi(t) \omega (t) dA(t) \\
= & \frac{2}{r^2} \widehat{\psi}_r (z).
\end{align*}
So $ \widehat{\psi}_r (z) \in L^{p/2}(\mathbb{D}, d \lambda)$ implies $ \psi(z) \in L^{p/2}(\mathbb{D}, d \lambda)$. Now, suppose that
$ \psi(z) \in L^{p/2}(\mathbb{D}, d \lambda)$. From the argument above, noting that $(1-|t|^2) \asymp (1-|z|^2) \asymp |1-\overline{t}z| $ and
$\frac{\omega(t)}{\omega(z)} \leq c$, for $t \in \Delta(z,r)$, we have
\begin{align*}
\widehat{\psi}_r (z)^{p/2} \leq & \frac{c}{\omega(z)^{p/2}} \sup \{\psi(t)^{p/2} \omega(t)^{p/2}: t \in \Delta(z,r) \} \\
\leq & \frac{c}{\omega(z)^{p/2}} \sup_{t \in \Delta(z,r)}  \int_{\Delta(t,r)} \psi(s)^{p/2} \omega(s)^{p/2} dA(s) \\
\leq & \frac{c}{\omega(z)^{p/2}} \sup_{t \in \Delta(z,r)}  \int_{\Delta(t,r)} \frac{(1-|z|^2)^2}{|1-\overline{s} z|^4}
\psi(s)^{p/2} \omega(s)^{p/2} dA(s).
\end{align*}
Since $t \in \Delta(z,r)$, we can choose $\Delta(t,r)$ so that $\Delta(t,r) \subset \Delta(z,r)$. Then
\begin{align*}
 \widehat{\psi}_r (z)^{p/2} \leq & \frac{c}{\omega(z)^{p/2}} \int_{\Delta(z,r)} \frac{(1-|z|^2)^2}{|1-\overline{s} z|^4}
\psi(s)^{p/2} \omega(s)^{p/2} dA(s) \\
\leq & c \int_{\mathbb{D}} \frac{(1-|z|^2)^2}{|1-\overline{s} z|^4}
\psi(s)^{p/2} dA(s).
\end{align*}
By Fubini's Theorem and well known theorem (Theorem 1.12 \cite{zhu2}), we get
$$\int_{\mathbb{D}} \widehat{\psi}_r (z)^{p/2}  d\lambda (z) \leq c \int_{\mathbb{D}} \psi (s)^{p/2}  d\lambda (s).$$
\end{proof}
If $p=2$, then we have another characterization for Hilbert-Schmidt composition operators.
\begin{cor} \label{c2}
Let $\omega$ be an admissible  weight satisfy (L1) condition. Then  $C_{\varphi} $ is Hilbert-Schmidt on $\mathcal{H}_{\omega}$
if and only if
$$ \int_{\mathbb{D}} \frac{N_{\varphi, \omega} (z)}{\omega(z) (1-|z|^2)^2} dA(z) =
\int_{\mathbb{D}} \frac{N_{\varphi, \omega} (z)}{\omega(z)} d\lambda(z) < \infty. $$
\end{cor}
\section{ \bf Closed Range}
It is well known that having the closed range for a bounded operator acting on a Hilbert space $H$ is equivalent to existing a positive constant
$c$ such that for every $f \in H$, $||Tf||_H \geq c ||f||_H$. Consider the function
$$ \tau_{\varphi, \omega} (z) =  \frac{N_{\varphi, \omega} (z)}{\omega(z)}. $$
\begin{prop} \label{g5}
Let $\omega$ be an admissible  weight  and  $C_{\varphi}$ be a bounded operator on $\mathcal{H}_{\omega}$. Then
$C_{\varphi}$ has closed range if and only if there exists a constant $c > 0$ such that for all $f \in \mathcal{H}_{\omega}$
\begin{equation} \label{r3}
\int_{\mathbb{D}} |f'(z)|^2  \tau_{\varphi, \omega} (z) \omega (z) \ dA(z)  \geq c
 \int_{\mathbb{D}} |f'(z)|^2  \omega (z) \ dA(z).
\end{equation}
\end{prop}
\begin{proof}
If $\varphi(0)=0$, the we can consider $C_{\varphi}$ acting on $\dot{\mathcal{H}_{\omega}}$, the closed subspace
of $\mathcal{H}_{\omega}$ consisting all functions with $f(0)=0$. Note that $C_{\varphi}$ has closed range
if and only if there exists a constant $c > 0$ such that  $||C_{\varphi}f||_{\mathcal{H}_{\omega}} \geq ||f||_{\mathcal{H}_{\omega}}$. But
\begin{align*}
||C_{\varphi}f||_{\mathcal{H}_{\omega}}^2 = & \int_{\mathbb{D}} |f'(\varphi(z))|^2 |\varphi'(z)|^2 \omega(z) \ dA(z) \\
= & \int_{\mathbb{D}} |f'(z)|^2 N_{\varphi, \omega} \ dA(z) \\
= & \int_{\mathbb{D}} |f'(z)|^2 \tau_{\varphi, \omega} (z) \omega(z) \ dA(z).
\end{align*}
Thus, in this case the proposition is proved. If $\varphi(0) = a \not = 0$, define the function $\psi = \sigma_a \circ \varphi$. Then
$C_{\varphi} = C_{\psi} C_{\sigma_a}$ and $C_{\sigma_a}$ is invertible on $\mathcal{H}_{\omega}$. Therefore $C_{\varphi}$ has closed range
if and only if $C_{\psi}$ has closed range. Since $\psi (0) = 0$, the argument above shows that  $C_{\psi}$ has closed range
if and only if there exists a constant $c > 0$ such that
\begin{equation} \label{r4}
\int_{\mathbb{D}} |f'(z)|^2  \tau_{\psi, \omega} (z) \omega (z) \ dA(z)  \geq c
 \int_{\mathbb{D}} |f'(z)|^2  \omega (z) \ dA(z).
\end{equation}
We just prove that (\ref{r3}) and (\ref{r4}) are equivalent. If (\ref{r3}) holds, then
\begin{align*}
\int_{\mathbb{D}} |f'(z)|^2  \tau_{\psi, \omega} (z) \omega (z) \ dA(z) =&  \int_{\mathbb{D}} |f'(z)|^2 N_{\psi, \omega}  \ dA(z) \\
=& \int_{\mathbb{D}} |(f \circ \psi)'(z)|^2 \omega (z) \ dA(z) \\
= & \int_{\mathbb{D}}  |(f \circ \sigma_a)'(\varphi(z))|^2 |\varphi'(z)|^2 \omega (z) \ dA(z)\\
= & \int_{\mathbb{D}}  |(f \circ \sigma_a)'(z)|^2 \tau_{\varphi, \omega} (z) \omega (z) \ dA(z) \\
\geq & c \int_{\mathbb{D}}  |(f \circ \sigma_a)'(z)|^2  \omega (z) \ dA(z) \\
= & c  \int_{\mathbb{D}}  |f'(z)|^2  \omega (\sigma_a(z)) \ dA(z) \\
\asymp & c \int_{\mathbb{D}}  |f'(z)|^2  \omega (z) \ dA(z).
\end{align*}
The last equation is due to Lemma \ref{l2}.
Hence (\ref{r4}) holds. Since $\varphi = \sigma_a \circ \psi$, the proof of converse part is similar.
\end{proof}
Fredholm composition operator is an example of composition operator with closed range property.
Recall that a bonded operator $T$ between two Banach spaces $X,Y$ is called Fredholm if Kernel $T$ and
$T^*$ are finite dimensional.
\begin{exm}
Suppose that $C_{\varphi}: \mathcal{H}_{\omega} \rightarrow \mathcal{H}_{\omega}$ be a Fredholm operator. By
Theorem 3.29\cite{cown}, $\varphi$ is an authomorphism of $\mathbb{D}$. Then $N_{\varphi, \omega} (z) = \omega (\varphi^{-1} (z))$.
If $\varphi(0)=0$, Schwarz Lemma implies that $|\varphi^{-1} (z)| \leq |z|$. Since $\omega$ is non-increasing,
$\omega (\varphi^{-1} (z)) = \omega (|\varphi^{-1} (z)|) \geq \omega (|z|) = \omega (z)$. Now (\ref{r3}) holds.
If $\varphi(0) \not = 0$, then the same argument can be applied.
\end{exm}

\subsection*{Acknowledgment}


\begin{thebibliography}{00}
\bibitem{cown}
C. C. Cowen and B. D. Maccluer, { \it Composition operators on spaces of analytic functions}, Studies in
Advanced Mathematics, CRC Press, Boca Raton, Fla, USA, 1995.

\bibitem{kellay1}
K. Kellay and P. Lef\`{e}vre, {\it Compact composition operators on weighted Hilbert spaces of analytic functions}, J.
Math. Anal. Appl. \textbf{386} (2012), 718-727.


\bibitem{lin}
M. Lindstr\"{o}m, E. Wolf, {\it Essential norm of the difference of weighted composition operators},
Monatsh. Math. \textbf{153} (2008), 133-143.

\bibitem{lues}
D. Luecking, {\it Trace ideal criteria for Toeplitz operators}, J. Funct. Anal. \textbf{73} (1987), 345-368.

\bibitem{lus}
W. Lusky, {\it On weighted spaces of harmonic and holomorphic functions}, J. London Math.
Soc. \textbf{51} (1995), 309-320.


\bibitem{pau}
J. Pau, P.A. P\'{e}rez, {\it Composition operators acting on weighted
Dirichlet spaces}, J. Math. Anal. Appl. \textbf{401} (2012), 682-694.

\bibitem{shap}
J. H. Shapiro, { \it Composition operator and classical function theory}, Springer-Verlag, New
York, 1993.


\bibitem{moafa}
M. Wang, {\it Weighted  composition operators
between Dirichlet spaces}, Acta Math. Sci. \textbf{31B}(2) (2011), 641-651.

\bibitem{zhu1}
K. Zhu, {\it Schatten class composition operators on weighted Bergman spaces of the disk},
J. Operator Theory \textbf{46} (2001), 173-181.

\bibitem{zhu2}
K. Zhu, {\it Spaces of holomorphic functions in the unit
ball}, Springer, New York, 2005.

\bibitem{zhu3}
K. Zhu, {\it Operator Theory in Function Spaces}, Second Edition, Math. Surveys and Monographs, Vol. 138,
American Mathematical Society: Providence, Rhode Island, 2007.

\end{thebibliography}
\end{document}